\documentclass[12pt,reqno]{amsart}
\usepackage[pagebackref,colorlinks=true]{hyperref}
\usepackage{verbatim}
\usepackage{eucal,url,amssymb,stmaryrd,enumerate,amscd,}
\usepackage{amsfonts,amsmath,amsthm,amssymb,amscd,enumerate,eucal,url,stmaryrd}
\usepackage{mathtools}
\usepackage[margin=1in]{geometry}
\usepackage{xcolor}
\usepackage{tikz}
\usetikzlibrary{arrows.meta,positioning}

\usepackage{mathrsfs}

\numberwithin{equation}{section}

\newcommand{\R}{\mathbb{R}}

\newcommand{\Rp}{\R_{>0}}

\makeatother
\newtheorem{Assumption}{Assumption}[section]
\newtheorem{Theorem}{Theorem}[section]
\newtheorem{Lemma}{Lemma}[section]
\newtheorem{Proposition}{Proposition}[section]
\newtheorem{Corollary}{Corollary}[section]
\newtheorem{Definition}{Definition}[section]
\newtheorem{Remark}{Remark}[section]
\newtheorem{Example}{Example}[section]

\usepackage{amsthm}

\newtheoremstyle{axiomstyle}
  {}{}                 
  {\itshape}           
  {}                   
  {\bfseries}          
  {.}                  
  {0.5em}              
  {\thmname{#1}~\thmnumber{#2}\thmnote{ (#3)}} 
  
\theoremstyle{axiomstyle}

\newcommand{\C}{\mathbb{C}}

\begin{document}

\newcommand{\add}[1]{{\color{blue}#1}}
\newcommand{\del}[1]{{\color{red}#1}}
\newenvironment{added}{\begingroup\color{blue}}{\endgroup}
\newenvironment{deleted}{%
  \begingroup\color{red}%
  \renewcommand{\cite}[2][]{\relax}%
  \renewcommand{\label}[1]{\relax}%
  \let\ref\relax
  \let\eqref\relax
  \let\pageref\relax
  \renewcommand{\section}[1]{\par\medskip\noindent{\bfseries [deleted section] ##1}\par}%
  \renewcommand{\subsection}[1]{\par\medskip\noindent{\bfseries [deleted subsection] ##1}\par}%
  \renewcommand{\subsubsection}[1]{\par\medskip\noindent{\bfseries [deleted subsubsection] ##1}\par}%
  \renewcommand{\paragraph}[1]{\par\noindent{\bfseries [deleted paragraph] ##1}\par}%
  \renewcommand{\subparagraph}[1]{\par\noindent{\bfseries [deleted subparagraph] ##1}\par}%
  \renewenvironment{equation}{\[\ignorespaces}{\]\ignorespacesafterend}%
  \renewenvironment{theorem}[1][]{\par\medskip\noindent\textbf{[deleted theorem] }\ignorespaces}{\par\medskip}%
  \renewenvironment{lemma}[1][]{\par\medskip\noindent\textbf{[deleted lemma] }\ignorespaces}{\par\medskip}%
  \renewenvironment{proposition}[1][]{\par\medskip\noindent\textbf{[deleted proposition] }\ignorespaces}{\par\medskip}%
  \renewenvironment{corollary}[1][]{\par\medskip\noindent\textbf{[deleted corollary] }\ignorespaces}{\par\medskip}%
  \renewenvironment{definition}[1][]{\par\medskip\noindent\textbf{[deleted definition] }\ignorespaces}{\par\medskip}%
  \renewenvironment{remark}[1][]{\par\medskip\noindent\textbf{[deleted remark] }\ignorespaces}{\par\medskip}%
  \renewenvironment{example}[1][]{\par\medskip\noindent\textbf{[deleted example] }\ignorespaces}{\par\medskip}%
  \renewenvironment{notation}[1][]{\par\medskip\noindent\textbf{[deleted notation] }\ignorespaces}{\par\medskip}%
  \renewenvironment{convention}[1][]{\par\medskip\noindent\textbf{[deleted convention] }\ignorespaces}{\par\medskip}%
  \renewenvironment{axiom}[1][]{\par\medskip\noindent\textbf{[deleted axiom] }\ignorespaces}{\par\medskip}%
}{\endgroup}

\newcommand{\Poly}{\R[u,v]}

\begin{abstract}
We study functions satisfying the composition law $F(xy)+F(x/y)=P(F(x),F(y))$ with a symmetric polynomial combiner $P$.
We prove that symmetry together with a quadratic degree bound on $P$ {forces a composition
law of d'Alembert type.} We establish a degree mismatch exclusion criterion showing that symmetric polynomial combiners with {$\deg P(u,v) \ge 3$} do not admit nonconstant continuous solutions, provided the leading term does not cancel (Theorem~\ref{thm:degree-exclusion}).
For continuous nonconstant functions $F:\mathbb{R}_{>0}\to\mathbb{R}$ with $F(1)=0$
satisfying  the composition law with a symmetric polynomial $P$ of degree at most two,
the combiner is necessarily of the form $P(u,v)=2u+2v+c\,uv$, $c\in\mathbb{R}$ (Theorem~\ref{thm:bilinear-forced}). The equation reduces in logarithmic coordinates to the classical d'Alembert functional equation.
For $c\neq 0$, one obtains hyperbolic or trigonometric branches,
while $c=0$ yields the squared-logarithm family.
Under the cost-function assumptions $F\ge 0$ and convexity,
only the hyperbolic branch with $c>0$ remains.
A unit log-curvature calibration selects the canonical value $c=2$,
which yields the canonical reciprocal cost $F(x)=\tfrac12(x+x^{-1})-1$.
For $c\neq0$, the result extends to $\mathbb{R}_{>0}^n$: every solution depends only on a single linear combination of coordinate logarithms; for $c=0$, the solution is a general quadratic form $\sum_{i,j}a_{ij}\ln x_i\ln x_j$.
In either case, nontrivial coordinate-wise separable costs are excluded.


%




\end{abstract}

\title[The d'Alembert Inevitability Theorem]{The d'Alembert Inevitability Theorem}
\date{\today}

\author{Jonathan Washburn}
\address[Jonathan Washburn]{Recognition Physics Institute Austin, Texas, USA}
\email{jon@recognitionphysics.org}
\author{Milan Zlatanovi\'c}
\address[Milan Zlatanovi\'c]{Department of Mathematics, Faculty of Science and Mathematics, University of Ni\v s, Vi\v segradska 33, 18000 Ni\v s, Serbia}
\email{zlatmilan@yahoo.com}

\author{Elshad Allahyarov}
\address[Elshad Allahyarov]{Recognition Physics Institute, Austin, TX, USA \\Institut für Theoretische Physik II: Weiche Materie, Heinrich-Heine-Universität Düsseldorf, Germany \\ 
Theoretical Department, Joint Institute for High Temperatures, RAS, Moscow, Russia \\  Department of Physics, Case Western Reserve University, Cleveland, OH, USA }
\email{elshad.allakhyarov@case.edu}

\maketitle 

\setcounter{tocdepth}{3}


\newcommand{\config}{\mathcal{C}}
\newcommand{\configR}{\mathcal{C}_R}

\section{Introduction}

Functional equations often arise when one requires that a quantity associated with ratios behaves consistently under multiplicative composition. 
Such consistency principles appear naturally in many contexts, including the theory of {functional equations} 
~\cite{Aczel,Davison, EbanksStetkaer, Hyers,Kuczma}, information geometric theory involving multiplicative models, and~models of ratio-based~costs.

The classical d'Alembert functional {equation} 

\[
H(t+u)+H(t-u)=2H(t)H(u)
\]
{is one of} 
 the central equations in the theory of functional equations. {The origin goes back to d'Alembert's derivation of the parallelogram law of forces~\cite{dAlembert1769}. Poisson~\cite{Poisson1804} gave a rigorous treatment of solutions, and~Picard~\cite{Picard1922} studied its relation with non-Euclidean geometry.}
Its continuous solutions are the cosine-type functions $\cosh(\alpha t)$
and $\cos(\alpha t)$ (see~\cite{Aczel,Kuczma,StetkaerBook}).
Many nonlinear functional relations reduce to this equation after suitable~transformations.

In~\cite{WZ1}, a~rigidity result for 
\( F : \mathbb{R}_{>0} \to \mathbb{R} \) is obtained.
Assuming the polynomial composition law
\begin{equation}\label{clp}
F(xy)+F\Big(\frac xy\Big )= 2F(x)F(y)+2F(x)+2F(y),
\end{equation}
together with the curvature calibration
\[
\lim_{t \to 0} \frac{2F(e^t)}{t^2} = 1,
\]
the function \( F \) is uniquely determined.
The unique solution is the canonical reciprocal cost
\[
F(x)=\tfrac12\bigl(x+x^{-1}\bigr)-1.
\]

This raises a natural structural question.
Is the composition law \eqref{clp} merely a modeling assumption, {as it appears in various contexts in the literature (see, e.g.,~\cite{AczelDhombres,StetkaerBook}),}
or is it forced by more general consistency requirements?

In this paper, we study functional relations of the form
\begin{equation}\label{inev}
F(xy)+F\!\left(\frac{x}{y}\right)
=
P\big(F(x),F(y)\big),
\qquad x,y>0,
\end{equation}
where $P$ is a polynomial combiner. 
We investigate which polynomial laws admit nontrivial continuous~solutions.

In applications, one often interprets $F(x)$ as a cost or penalty associated
with a ratio $x>0$.  { 
The normalization $F(1)=0$ reflects that the identity element carries zero deviation. 
It does not restrict generality, since any solution can be reduced to this case by subtracting a constant, with~a corresponding translation of the variables in $P$.
More precisely, if~$\widetilde{F}(x):=F(x)-F(1)$, then $\widetilde{F}(1)=0$ and
\[
\widetilde{F}(xy)+\widetilde{F}\!\left(\frac{x}{y}\right)
=
\widetilde{P}\bigl(\widetilde{F}(x),\widetilde{F}(y)\bigr),
\qquad
\widetilde{P}(u,v):=P\bigl(u+F(1),v+F(1)\bigr)-2F(1).
\]
Thus, the~assumption $F(1)=0$ is without loss of generality.} This condition determines the boundary identities
\[
P(u,0)=2u, \qquad P(0,v)=2v,
\]
which restrict the form of the polynomial combiner.
Requiring compatibility with multiplicative composition then leads 
to a d'Alembert-type functional equation on $\Rp$.

{ We assume that $F:\mathbb{R}_{>0}\to\mathbb{R}$ is continuous and nonconstant, and~that $P\in\mathbb{R}[u,v]$ is a symmetric polynomial, where $\mathbb{R}[u,v]$ denotes the ring of polynomials in two variables with real coefficients.}
Symmetry is natural since the roles of $x$ and $y$
in the left-hand side of~\eqref{inev} are interchangeable.
Under symmetry of $P$, we derive reciprocity 
$F(x)=F(1/x)$.

Our first result treats the case of higher-degree polynomial combiners.
We show that symmetric polynomial combiners of degree $d\ge3$
are incompatible with the functional Equation \eqref{inev}.
More precisely, if~$P\in\R[u,v]$ is symmetric, satisfies $P(0,v)=2v$,
and its leading term does not cancel on the diagonal,
then Equation \eqref{inev}
admits no continuous nonconstant solution $F:\Rp\to\R$ with $F(1)=0$.
Consequently, only polynomial combiners of degree at most two
can admit nontrivial continuous solutions.
This reduction to the quadratic case is the main structural step of the~paper.

Our main structural result shows that in the quadratic case
the composition law is completely determined.
If $F:\Rp\to\R$ is continuous and nonconstant and the combiner $P$
is a symmetric polynomial of degree at most two, then necessarily
\[
P(u,v)=2u+2v+c\,uv,
\qquad c\in\R.
\]
Under the normalization $F(1)=0$, the functional equation therefore reduces to
\[
F(xy)+F\Big(\frac xy\Big)
=
2F(x)+2F(y)+c\,F(x)F(y).
\]

Passing to logarithmic coordinates reduces this relation to the classical d'Alembert equation
\[
H(t+u)+H(t-u)=2H(t)H(u),
\]
whose continuous solutions are well known
~\cite{Aczel,AczelDhombres,Czerwik,Kannappan2,Kuczma,Papp,StetkaerBook}.
All continuous solutions of the original equation can therefore be described~explicitly.

Convexity and nonnegativity select the hyperbolic branch {(Corollary~\ref{cor:convex-case})},
while a curvature normalization determines the distinguished value $c=2$.
In this case, the canonical reciprocal cost
\[
F(x)=\dfrac12(x+x^{-1})-1
\]
appears as a structurally determined solution~\cite{WZ1}.

{ Finally, we extend the analysis to functions on $\mathbb{R}_{>0}^n := \{(x_1,\dots,x_n)\in\mathbb{R}^n : x_i > 0\}$.
In the multidimensional case, passing to logarithmic coordinates reduces the problem to a functional equation on $\mathbb{R}^n$ involving sums and differences. 
Such equations are known to exhibit a collapse to one-dimensional dependence, as~in the classical case (see, e.g.,~\cite{AczelDhombres,StetkaerBook}). 
We show that, for~$c\neq 0$, solutions {depend on} 
 $\mathbf{x}$ only through the scalar quantity $\boldsymbol{\alpha}\cdot \ln \mathbf{x}$,  where $\mathbf{x}=(x_1,\dots,x_n)\in\mathbb{R}_{>0}^n$ denotes a vector variable, and~$\boldsymbol{\alpha}=(\alpha_1,\dots,\alpha_n)\in\R^n$ is a vector of weights.
Thus, the effective dependence remains one-dimensional.}

The paper is organized as follows.
In Section~\ref{sec:preliminaries},
we study structural consequences of the polynomial composition law.
In particular, we prove reciprocity of $F$ under symmetry of the combiner
and derive the boundary identities that restrict the polynomial $P$.

Section~\ref{sec:polynomial-classification}
contains the classification of admissible polynomial combiners.
We show that symmetric combiners of degree at least three
do not admit nonconstant continuous solutions under a natural
non-cancellation assumption.
In the quadratic case, we obtain the bilinear form of the combiner. In~Section~\ref{sec:reduction},
we pass to logarithmic coordinates and reduce the equation
to the classical d'Alembert functional equation.
Using the known classification of its continuous solutions,
we obtain the corresponding families of functions $F$.

{In Section~\ref{sec:Ndim},
we consider the multidimensional case.
We show that for $c\neq0$, every solution depends only on
the scalar quantity $\boldsymbol{\alpha}\cdot\ln\mathbf{x}$. Thus, even in dimension $n$, the effective dependence
is one-dimensional through the quantity
$\boldsymbol{\alpha}\cdot\ln\mathbf{x}$.
 We give an explicit \mbox{16-dimensional}
example (Example~\ref{ex:16dim}) illustrating the collapse to a single logarithmic~direction.}

In the final section, we introduce a normalization
based on the log-curvature
\(
\kappa(F).
\)
We show that this calibration fixes the parameter of the bilinear
composition law.
For convex nonnegative solutions with $\kappa(F)=1$,
the parameter is uniquely determined by $c=2$,
which yields the canonical reciprocal~cost.

{ \medskip

\noindent {\bf {Main contributions of the paper.}
}
The main results of the paper are the~following:
\begin{itemize}
\item[(i)] We show that symmetric polynomial combiners of degree $d\ge 3$ do not admit nonconstant continuous solutions under a natural non-cancellation assumption.
\item[(ii)] We prove that in the case $\deg P\le 2$, the~polynomial combiner is of bilinear d'Alembert type, namely $P(u,v)=2u+2v+c\,uv$.
\item[(iii)]  We reduce the functional equation to the classical d'Alembert equation and give a complete classification: all continuous solutions form three families: the hyperbolic branch $F(e^t)=\tfrac{2}{c}(\cosh(\alpha t)-1)$, the~trigonometric branch $F(e^t)=\tfrac{2}{c}(\cos(\alpha t)-1)$ (for $c\neq 0$), and~the quadratic-logarithm family $F(x)=k(\ln x)^2$ (for $c=0$).
\item[(iv)] We extend the result to the $n$-dimensional case and show that every solution depends only on a single linear combination of the logarithmic variables.
\end{itemize} 

The main difficulty lies in the analysis of the equation with a general polynomial combiner, where a degree-based argument and a non-cancellation condition are used to exclude higher-degree cases.}

\section{Preliminaries}\label{sec:preliminaries}

We consider a function $F:\Rp\to\R$.
Passing to logarithmic coordinates, we define
\begin{equation}\label{eq:def-G}
G(t):=F(e^t), \qquad t\in\R.
\end{equation}

\begin{Definition}\label{def:poly-law}
A function $F:\Rp \to \R$ is said to satisfy a {polynomial composition law} 
if there exists a polynomial $P \in \R[u,v]$ such that
\begin{equation}\label{eq:poly-law}
F(xy) + F\!\left(\frac{x}{y}\right)
=
P\big(F(x), F(y)\big)
\qquad \text{for all } x,y>0.
\end{equation}
Such a polynomial $P$ is called a {combiner}.
\end{Definition}

In the following, we assume that \eqref{eq:poly-law} holds
with a symmetric polynomial $P \in \R[u,v]$. In~logarithmic coordinates $x=e^{t}$, $y=e^{u}$,
Equation \eqref{eq:poly-law} becomes
\begin{equation}\label{eq:log-law}
G(t+u)+G(t-u)
=
P\big(G(t),G(u)\big),
\qquad t,u\in\R.
\end{equation}
Note that \eqref{eq:poly-law} determines $P$
only on the subset
\[
\{(F(x),F(y)) : x,y \in \mathbb{R}_{>0}\}
\subseteq \mathbb{R}^2,
\]
that is,  on~$\operatorname{Range}(F)\times\operatorname{Range}(F)$.

We assume that $F$ is continuous on $\mathbb{R}_{>0}$ and nonconstant.
Then $\operatorname{Range}(F)$ contains a nontrivial interval.
Consequently, if~a polynomial $P(u,v)$ vanishes on 
$\operatorname{Range}(F)\times\operatorname{Range}(F)$,
then $P$ vanishes identically on $\mathbb{R}^2$.

\begin{Remark}\label{rem:regularity} Continuity is used to ensure that $\operatorname{Range}(F)$
contains a nondegenerate interval.
All results, such as reciprocity of $F$, the~boundary conditions, and~the bilinear form of $P$,  remain valid if $\operatorname{Range}(F)$ is dense in some interval.
In particular, continuity can be replaced by measurability together with local boundedness, since measurable solutions of the d'Alembert equation are continuous~\cite{Akkouchi,Kannappan2, StetkaerBook}.
\end{Remark}

\subsection{{Structural} 
 Properties}\label{sec:first-forcing-layer}

We now derive structural consequences of the polynomial composition law \eqref{eq:poly-law}.
We  establish the equivalence between symmetry of $P$ and reciprocity of $F$. The~normalization $F(1)=0$ implies that the boundary values of $P$ satisfy
\(
P(0,v)=2v\) and \(
P(u,0)=2u.
\)
{ The symmetry assumption $P(u, v) = P(v, u)$ is not arbitrary. It is the algebraic counterpart of~reciprocity.} 

As shown in the next lemma, symmetry of the combiner
implies reciprocal symmetry of $F$.
This expresses invariance under inversion of the ratio,
a natural structural property for a discrepancy~measure.

\begin{Lemma}\label{lem:swap-quotients}
If $F$ satisfies Equation \eqref{eq:poly-law} with a symmetric combiner $P$, then
\[
F(z)=F\Big(\frac{1}{z}\Big)
\qquad\text{for all}\; z>0.
\]
\end{Lemma}

\begin{proof}
Writing Equation \eqref{eq:poly-law} for $(x,y)$ and $(y,x)$ gives
\[
F(xy)+F\Big(\frac xy\Big)=P(F(x),F(y)),
\]
\[
F(yx)+F\Big(\frac yx\Big)=P(F(y),F(x)).
\]
Since $xy=yx$ and $P$ is symmetric, subtracting yields
\[
F\Big(\frac xy\Big)=F\Big(\frac yx\Big).
\]
If we set $y=1$ in the last equation, we obtain $F(x)=F(1/x)$, for~all $x>0$.
\end{proof}

\begin{Lemma}\label{lem2}
Assume $F$ satisfies Equation \eqref{eq:poly-law} where $P$ is a polynomial.
If $F$ is continuous, nonconstant, and~reciprocal-symmetric, then $P$ is symmetric.
\end{Lemma}

\begin{proof}
Using reciprocity, we obtain
\[
P(F(x),F(y))=P(F(y),F(x))
\qquad \text{for all }x,y>0.
\]
Let $Q(u,v)=P(u,v)-P(v,u)$.
Then $Q(F(x),F(y))=0$.
Since $\operatorname{Range}(F)$ contains an interval,
$Q$ vanishes on an open subset of $\R^2$,
hence $Q\equiv 0$.
\end{proof}

\begin{Remark} Under the polynomial assumption and nondegeneracy of the range (e.g., $F$ continuous and nonconstant),
symmetry of $P$ and reciprocity of $F$ are equivalent.
\end{Remark}

\medskip
With reciprocity established, the~next lemma determines the boundary values of $P$ at~$F(1)$.

\begin{Lemma}\label{lem:axis-at-F1}
Let $F$ be continuous and nonconstant satisfying Equation \eqref{eq:poly-law}
with symmetric $P$.
Then
\[
P(F(1),v)=2v,
\qquad
P(u,F(1))=2u
\quad\text{for all }u,v\in\R.
\]
\end{Lemma}

\begin{proof}
Setting $x=1$ in Equation \eqref{eq:poly-law} gives
\[
F(y)+F\Big(\frac 1y\Big)=P(F(1),F(y)).
\]
By reciprocity, $F(y)=F(1/y)$, hence
\[
2F(y)=P(F(1),F(y)).
\]
Let us define the polynomial $q(v)=P(F(1),v)-2v$.
Then $q(F(y))=0$.
Since the range of $F$ contains an interval, $q\equiv 0$.
Symmetry gives the second identity.
\end{proof}

\begin{Corollary}\label{lem:boundary-conditions}
Under the assumptions of Lemma \ref{lem:axis-at-F1} and the additional condition $F(1)=0$, it follows~that
\[
P(u,0)=2u,
\qquad
P(0,v)=2v.
\]
\end{Corollary}

\medskip

We now factor out the constraints at $F(1)$.

\begin{Lemma}\label{lem:shift-factor-F1}
Let $F$ satisfy Equation \eqref{eq:poly-law} with polynomial symmetric $P$.
Assume $F$ is continuous and nonconstant.
Then there exists $R\in\R[u,v]$ such that
\begin{equation}\label{eddd}
P(u,v)
=
2u+2v-2F(1)
+
(u-F(1))(v-F(1))\,R(u,v).
\end{equation}
\end{Lemma}

\begin{proof}
Let us define polynomial
\[
S(u,v)=P(u,v)-2u-2v+2F(1).
\]
By Lemma~\ref{lem:axis-at-F1}, we have
\[
S(F(1),v)=0,
\qquad
S(u,F(1))=0.
\]
Fix $v$. Since $S(F(1),v)=0$,
it follows that $(u-F(1))$ divides $S(u,v)$.
Hence, there exists a polynomial $T\in\R[u,v]$ such that
\[
S(u,v)=\big(u-F(1)\big)\,T(u,v).
\]
Now evaluate $S(u,v)$ at $v=F(1)$. We obtain
\[
0=S(u,F(1))
=\big(u-F(1)\big)\,T(u,F(1))
\qquad\text{for all }u.
\]
Since this identity holds for all $u$, it follows that
\[
T(u,F(1))=0 \quad \text{for all }u.
\]
Thus, $(v-F(1))$ divides $T(u,v)$.
Therefore, there exists $R\in\R[u,v]$ such that
\[
T(u,v)=\big(v-F(1)\big)\,R(u,v).
\]

Combining the two factorizations yields
\[
S(u,v)
=
\big(u-F(1)\big)\big(v-F(1)\big)\,R(u,v),
\]
which proves Equation \eqref{eddd}.
\end{proof}

\begin{Corollary}\label{cor:factor-F1-zero}
Under the assumptions of Lemma \ref{lem:shift-factor-F1} and the additional condition $F(1)=0$,
there exists a polynomial $R\in\R[u,v]$ such that
\[
P(u,v)=2u+2v+uv\,R(u,v)
\qquad \text{for all }u,v\in\R.
\] 
\end{Corollary}
Since $P$ is symmetric, it follows that $R$ is also~symmetric.
 \section{Polynomial~Classification}\label{sec:polynomial-classification}

In this section, we classify the possible polynomial combiners $P$.



Using the factorization~\eqref{eddd} from Lemma~\ref{lem:shift-factor-F1}, 
together with the symmetry of $P$ and the boundary conditions from 
Section~\ref{sec:first-forcing-layer}, we conclude that the polynomial $R$ is symmetric. 
Indeed, the~term $2u+2v-2F(1)$ is already symmetric in $(u,v)$, and~therefore, the symmetry 
of $P$ forces $R$ in~\eqref{eddd} to be symmetric as well.
\medskip

At this point, the~degree of $R$ is not restricted.
If $R$ has high degree, then the functional equation becomes structurally more complex.
Using \eqref{eddd},~Equation \eqref{eq:log-law} becomes
\begin{align*}
G(t+u)+G(t-u)
=&
2G(t)+2G(u)-2F(1)
\\&+\big(G(t)-F(1)\big)\big(G(u)-F(1)\big)\,R\big(G(t),G(u)\big).
\end{align*}
Assume that $G$ is smooth in a neighborhood of $0$. Then both sides of the above equation admit Taylor expansions at $(t,u)=(0,0)$ in the form of  convergent power series
\[
G(t+u)+G(t-u)=\sum_{k,l\ge 0} A_{kl}\, t^k u^l,\]
\[2G(t)+2G(u)-2F(1)
+\big(G(t)-F(1)\big)\big(G(u)-F(1)\big)R\big(G(t),G(u)\big)
=\sum_{k,l\ge 0} B_{kl}\, t^k u^l.
\]
By uniqueness of power series expansions, we obtain $A_{kl}=B_{kl}$ for all $k,l\ge 0$. 
This yields an infinite system of algebraic relations between the derivatives of $G$ at $0$ and the coefficients of the polynomial $R$. Consequently, without~imposing a bound on the degree of $P$, the~classification problem leads to an infinite system of compatibility conditions.

For this reason, we restrict our attention to polynomial combiners
of total degree at most two. Since $(u-F(1))(v-F(1))$ already has degree two,
it follows that $R$ in (\ref{eddd}) must be~constant.

\begin{Assumption}\label{ass:degree-two}
The combiner $P$ has total degree at most two.
\end{Assumption}



\begin{Theorem}\label{thm:degree-exclusion}
Let $P\in\R[u,v]$ be a symmetric polynomial of degree $d\ge 3$ with $P(0,v)=2v$,
and let $q(y):=P(y,y)$. 
Assume that {$\deg_y q = d$, where the index $y$ indicates the variable with respect to which the degree is taken,}
and that no cancellation occurs in the leading term, so that
\[
\deg_y\bigl(P(P(q(y),y)-y,y)\bigr)=d^3-2d^2+2d.
\]
Then there is no continuous nonconstant function $F:\Rp\to\R$ with $F(1)=0$
satisfying the polynomial composition law
\[
F(xy)+F\Big(\frac xy\Big)=P(F(x),F(y)).
\]
In particular, the~explicit degree-three combiner treated in
Example~\ref{exmldeg3} is excluded.
\end{Theorem}

\begin{proof}
Set $G(t):=F(e^t)$, so that $G(0)=0$ and the functional equation becomes
\begin{equation}\label{eq:FE-G}
G(t+u)+G(t-u)=P(G(t),G(u)), \qquad t,u\in\R.
\end{equation}
Since $P(0,v)=2v$ and $P$ is symmetric, the~factored form
$P(u,v)=2u+2v+uv\,R(u,v)$ holds for some symmetric polynomial $R$ of degree $d-2\ge 1$.

{We observe that, for~each integer $n\ge 1$, the~quantity $G(ns)$ can be expressed as a polynomial function of $y=G(s)$. Setting $y=G(s)$ and $q(y):=P(y,y)$, the~composition law \eqref{eq:FE-G} gives
\[
G(2s)=P(G(s),G(s))=q(G(s))=q(y).
\]Setting $t=2s,\,u=s$ in \eqref{eq:FE-G}, we obtain
\[
G(3s) = P(G(2s), G(s)) - G(s) = P(q(y), y) - y.
\]
Setting $t=u=2s$ in \eqref{eq:FE-G}, we obtain
\[
G(4s) = q(G(2s)) = q(q(y)).
\]
Setting $t=3s,\,u=s$ in \eqref{eq:FE-G}, we obtain the identity
\begin{equation}\label{eq:key-identity}
G(4s) + G(2s) = P(G(3s), G(s)).
\end{equation}
Proceeding inductively, each $G(ns)$ is obtained from $y$ by finitely many polynomial substitutions involving $P$. Therefore, for~each fixed $n$, $G(ns)$ is a {polynomial in} 
 $y=G(s)$.} 
\end{proof}
 
We will now analyze this identity as a polynomial relation in $y=G(s)$.

\begin{Lemma}\label{lem:degree-mismatch}
Let $P\in\R[u,v]$ be a symmetric polynomial of degree $d\ge 3$ with $P(0,v)=2v$,
and let $q(y):=P(y,y)$.
Let $G(3s)$ and $G(4s)$ be the polynomials in $y=G(s)$ obtained in the proof
of Theorem~\ref{thm:degree-exclusion}.~Then
\begin{itemize}
\item[(i)] $\deg_y G(4s) = (\deg_y q)^2$.
\item[(ii)] $\deg_y\!\bigl(P(G(3s),G(s))\bigr)\le d^3-2d^2+2d$,
with equality under the non-cancellation assumption of
Theorem~\ref{thm:degree-exclusion}.
\item[(iii)] If $\deg_y q=d$ and equality holds in {\rm(ii)}, then the degree difference is
$d(d-1)(d-2)\ge6$ for all $d\ge3$.
\end{itemize}
\end{Lemma}

\begin{proof}
\begin{itemize}
\item[(i)] Since
\(
G(4s)=q(q(y)),
\)
and ${\deg_y q=d}$, we obtain
\(
\deg_y G(4s)=d^2.
\)
\item[(ii)] From the identities above, we have
\end{itemize}
\[
G(3s)=P(q(y),y)-y.
\]
Write
\[
P(u,v)=2u+2v+uv\,R(u,v),
\qquad \deg R(u,v)=d-2.
\]
The highest-degree contribution in $P(q(y),y)$ comes from
\(
q(y)\,y\,R(q(y),y).
\)
Since the degree of $R$ in the variable $u$ is at most $d-2$, and~$\deg_y q=d$, we have
\[
\deg_y R(q(y),y)\le (d-2)d=d^2-2d.
\]
Therefore
\[
\deg_y\bigl(q(y)\,y\,R(q(y),y)\bigr)\le d+1+(d^2-2d)=d^2-d+1.
\]
Hence
\[
\deg_y G(3s)\le d^2-d+1.
\]
Equality holds if the leading coefficient of 
$G(3s)$ (as a polynomial in $y$) is nonzero. This can be verified explicitly in the degree-three~case.

Now consider the right-hand side of \eqref{eq:key-identity}. Its terms are of the form
\[
{a_{ij}\,[G(3s)]^i\,y^j},
\qquad i+j\le d,\quad i,j\ge 1.
\]
The degree of such a term is
\[
i\,\deg_y G(3s)+j.
\]
Using the bound for $\deg_y G(3s)$, we obtain
\begin{align*}
\deg_y\bigl(P(G(3s),G(s))\bigr)
&\le
\max_{1\le i\le d-1}\bigl(i(d^2-d+1)+(d-i)\bigr)\\&=
\max_{1\le i\le d-1}\bigl(i(d^2-d)+d\bigr),
\end{align*}
and the maximum is achieved at $i=d-1$. Hence
\[
\deg_y\bigl(P(G(3s),G(s))\bigr)\le d^3-2d^2+2d.
\]
Under the non-cancellation assumption of Theorem~\ref{thm:degree-exclusion},
this is the   degree of the right-hand~side.

\begin{itemize}
\item[(iii)] By parts  {(i)} and  {(ii)}, in~the equality case, the degree difference is
\end{itemize}
\begin{align*}
(d^3-2d^2+2d)-d^2
&=
d^3-3d^2+2d
\\&=
d(d-1)(d-2).    
\end{align*}

Since $d\ge 3$, all three factors are positive, and~therefore
\[
d(d-1)(d-2)\ge 6.
\]
\end{proof}
\begin{proof}[{Final step of the proof of} 
 Theorem~\ref{thm:degree-exclusion}] 
Since $G$ is continuous and nonconstant, its range $\operatorname{Range}(G)$ contains a non-degenerate interval $I\subset \mathbb{R}$. From~\eqref{eq:key-identity}, we have
\[
G(4s)+G(2s)=P(G(3s),G(s)).
\]
By substituting $y=G(s)$, and~using the previous part of the proof, the~quantities $G(2s)$, $G(3s)$, and $G(4s)$ are polynomials in $y$. Therefore, the left-hand side and the right-hand side of \eqref{eq:key-identity} can be written as polynomial functions of $y$. More explicitly,
\[
A(y):=q(q(y))+q(y),
\qquad
B(y):=P\bigl(P(q(y),y)-y,y\bigr)
\]
belong to $\mathbb{R}[y]$, are independent of $s$, and~\eqref{eq:key-identity} becomes
\[
A(y)=B(y)
\qquad \text{for all } y\in \operatorname{Range}(G).
\]
Since $\operatorname{Range}(G)$ contains the interval $I$, it follows that
\[
A(y)-B(y)=0
\qquad \text{for all } y\in I.
\]
The polynomial  $A(y)-B(y)$ is a polynomial in one variable. A~polynomial that vanishes on a non-degenerate interval must vanish identically. Therefore
\[
A(y)\equiv B(y)
\qquad \text{on } \mathbb{R}.
\]
However, by~Lemma~\ref{lem:degree-mismatch}, the~two sides have different degrees, which is impossible. 
Hence, no continuous nonconstant function $F:\Rp\to\R$ with $F(1)=0$ satisfies the polynomial composition law for a symmetric polynomial $P$ of degree $d\ge 3$.
\end{proof}

{In Theorem~\ref{thm:degree-exclusion}, we assume that no cancellation occurs in the leading term. We now make this assumption~precise.

\begin{Remark}
Let $P\in\R[u,v]$ be symmetric of degree $d\ge 3$ and let $q(y)=P(y,y)$.
We assume that
\[
\deg_y q = d
\]
and
\[
\deg_y\bigl(P(P(q(y),y)-y,y)\bigr)=d^3-2d^2+2d.
\]
This ensures that the leading term is preserved under composition. 
Equivalently, there is no cancellation of the highest-degree contribution on the diagonal $u=v$. 
It guarantees that the right-hand side attains the maximal degree required for the degree mismatch argument.
\end{Remark}
}

\begin{Remark}
The diagonal polynomial $q(x)=P(x,x)$ may have degree $k<d$ if the
highest-degree terms of $P$ vanish on $u=v$.
For example,
\[
P(u,v)=2u+2v+uv(u-v)^2
\]
has degree $4$, while $q(x)=4x$.

For the combiner considered in
Example~\ref{exmldeg3}, we have
$q(x)=4x+2x^3$; hence, $\deg_y q=d=3$.
Thus, the possible degeneration $\deg_y q<d$ does not occur in the explicit case considered here.
\end{Remark}
{ The following corollary follows directly from Theorem~\ref{thm:degree-exclusion} under the assumptions clarified above.}
\begin{Corollary}\label{cor:degree-two-forced}
Let $F:\Rp\to\R$ be continuous and nonconstant, and~assume that
\[
F(xy)+F\Big(\frac xy\Big )=P(F(x),F(y)),
\]
where $P\in\R[u,v]$ is symmetric and $F(1)=0$.
Assume the non-cancellation assumption of
Theorem~\ref{thm:degree-exclusion}. Then $\deg P(u,v)\le 2.$
Consequently, by~Theorem~\ref{thm:bilinear-forced},
\[
P(u,v)=2u+2v+c\,uv
\]
is the unique polynomial composition law admitting nonconstant
continuous solutions.
\end{Corollary}

{\begin{proof}
If $\deg P\ge 3$, Theorem~\ref{thm:degree-exclusion} shows no continuous nonconstant $F$ with $F(1)=0$ can exist. Hence, $\deg P\le 2$. Then Theorem~\ref{thm:bilinear-forced} implies $P(u,v)=2u+2v+c\,uv$.
\end{proof}}

\label{sec:d3-verification}






%
%
%

\begin{Example}\label{exmldeg3}
Consider the polynomial
\[
P(u,v)=2u+2v+u^2v+uv^2 ,
\]
which has degree 3 and satisfies $P(0,v)=2v$.
Then
\[
q(x)=P(x,x)=4x+2x^3 .
\]
Let $y=G(s)$.   Using the identities derived in the proof of Theorem~\ref{thm:degree-exclusion}, we obtain
\begin{align*}
G(2s)&=4y+2y^3, \\
G(3s)&=9y+24y^3+18y^5+4y^7, \\
G(4s)&=16y+136y^3+192y^5+96y^7+16y^9 .
    \end{align*}
The identity
\[
G(4s)+G(2s)=P(G(3s),G(s))
\]
requires the equality of two polynomials in $y$.
We have
\[
\deg_y\bigl(G(4s)+G(2s)\bigr)=9,
\qquad
\deg_y\bigl(P(G(3s),G(s))\bigr)=15 .
\]
Thus, the degrees do not match, so the identity cannot hold identically.
\end{Example}

Under Assumption~\ref{ass:degree-two}, $P$ can be written in the general quadratic form
\begin{equation}\label{eq:quadratic-form}
P(u,v)=a+bu+cv+d\,uv+e\,u^2+f\,v^2,
\qquad a,b,c,d,e,f\in\R.
\end{equation}

\begin{Lemma}\label{lem:symmetry-reduction}
If $P$ is symmetric, i.e., $P(u,v)=P(v,u)$, then $b=c$ and $e=f$.
Consequently,
\begin{equation}\label{eq:quadratic-symmetric}
P(u,v)=a+b(u+v)+d\,uv+e(u^2+v^2).
\end{equation}
\end{Lemma}

\begin{proof}
The symmetry implies equality of coefficients after interchanging $u$ and $v$ in \eqref{eq:quadratic-form}.
Comparing the coefficients of $u$ and $v$ gives $b=c$, and~comparing those of $u^2$ and $v^2$ gives $e=f$.
\end{proof}

\medskip

We now determine the relations among the coefficients imposed by the functional equation.

\begin{Theorem}\label{thm:quadratic-relation}
Let $F:\Rp\to\R$ be continuous and nonconstant satisfying
\[
F(xy)+F\Big(\frac xy\Big )=P(F(x),F(y))
\qquad \text{for all }x,y>0,
\]
where $P$ is a symmetric quadratic polynomial of the form
\[
P(u,v)=a+b(u+v)+c\,uv+e(u^2+v^2).
\]
Then $e=0$, and~
\(
b(2-b)+ac=0.
\)
\end{Theorem}\begin{proof}
By Lemma~\ref{lem:axis-at-F1}, we have
\(
P(F(1),v)=2v\) for all \( v\in\R.
\)
Substituting $u=F(1)$ in the given form of $P$, we obtain
\[
a+b(F(1)+v)+cF(1)v+e(F(1)^2+v^2)=2v.
\]
Since this identity holds for all $v\in\R$, we have
\[
e=0,
\qquad
a+bF(1)=0,
\qquad
b+cF(1)=2.
\]
Eliminating $F(1)$  from the last two equations gives
\[
b(2-b)+ac=0.
\]
\end{proof}

We now consider the effect of the normalization at $x=1$.

\begin{Corollary}\label{cor:normalization}
According to the assumptions of Theorem~\ref{thm:quadratic-relation}, 
if $F(1)=0$, then
\[
a=0, \qquad b=2,
\]
and hence
\[
P(u,v)=2u+2v+c\,uv.
\]
\end{Corollary}

\begin{proof}
By Theorem~\ref{thm:quadratic-relation}, from~the relations $b+cF(1)=2$ and $a+bF(1)=0$, substituting $F(1)=0$ gives $b=2$ and $a=0$.
\end{proof}

\medskip

Thus, in~the case of $F(1)=0$, the~composition law reduces to
\begin{equation}\label{eq:F-bilinear-law}
F(xy)+F\Big(\frac xy\Big)
=
2F(x)+2F(y)+c\,F(x)F(y),
\qquad x,y>0.
\end{equation}
This equation will be analyzed in Section~\ref{sec:reduction},
where we make explicit its connection with the classical d'Alembert functional~equation.

\begin{Theorem}[d'Alembert Inevitability Theorem]\label{thm:bilinear-forced}
Let $F:\Rp\to\R$ be continuous and nonconstant satisfying a polynomial composition law with a symmetric combiner of degree at most two.
Then $P$ must be of the form
\[
P(u,v)=a+b(u+v)+c\,uv
\]
with $b(2-b)+ac=0$, {where $a,b,c\in\R$.}
If, moreover, $F(1)=0$, then
\[
P(u,v)=2u+2v+c\,uv,
\]
and $F$ satisfies Equation \eqref{eq:F-bilinear-law}.
\end{Theorem}

{
\begin{proof}
By Lemma~\ref{lem:symmetry-reduction}, the~polynomial $P$ has the form
\[
P(u,v)=a+b(u+v)+d\,uv+e(u^2+v^2).
\]
By Theorem~\ref{thm:quadratic-relation}, we have $e=0$; hence
\[
P(u,v)=a+b(u+v)+d\,uv.
\]
Renaming $d=c$, we obtain $P(u,v)=a+b(u+v)+c\,uv$. The~relation
\(
b(2-b)+ac=0
\)
follows directly from Theorem~\ref{thm:quadratic-relation}.

If $F(1)=0$, Corollary~\ref{cor:normalization} gives $a=0$ and $b=2$; hence
\[
P(u,v)=2u+2v+c\,uv.
\]
\end{proof}
}

\begin{Corollary}\label{cor:combiner-uniqueness}
Let $F:\Rp\to\R$ be continuous and nonconstant.
If two polynomials $P,P'\in\R[u,v]$ satisfy
\[
F(xy)+F\Big(\frac xy\Big)=P(F(x),F(y))=P'(F(x),F(y))
\qquad x,y>0,
\]
then $P=P'$.
\end{Corollary}

\begin{proof}
The identity implies $P(u,v)=P'(u,v)$ for all
$(u,v)\in\operatorname{Range}(F)\times\operatorname{Range}(F)$.
Since $F$ is continuous and nonconstant, its range contains a nondegenerate interval.
Hence, $P=P'$ on a set containing a rectangle in $\R^2$.
Therefore, the polynomial $P-P'$ is identically zero, and~thus $P\equiv P'$ on $\R^2$.
\end{proof}
\begin{Remark}{\label{prop:degree-one-trivial}
Let $F:\Rp\to\R$ be continuous and nonconstant, with~$F(1)=0$. If~$F$ satisfies \eqref{eq:poly-law} with a symmetric combiner $P\in\R[u,v]$ of total degree at most one, then $P(u,v)=2(u+v)$, and~the composition law coincides with \eqref{eq:F-bilinear-law} at $c=0$. Consequently, the~degree-one case is not a separate family; it is included in Theorem~\ref{thm:bilinear-forced} for $c=0$. Thus, degree two is the minimal degree for which a free parameter appears (namely $c$). }  
\end{Remark}

If, in~addition, $F$ is convex, then $x=1$ is a global minimum of $F$.
In this case, the~normalization $F(1)=0$ corresponds to shifting the minimum to~zero.

\begin{Lemma}\label{lem:symmetry-convex-min}
Let $F:\Rp\to\R$ be continuous and nonconstant, and~suppose
\[
F(xy)+F\Big(\frac xy\Big )=P(F(x),F(y)), \qquad x,y>0,
\]
where $P\in\R[u,v]$ is symmetric.
Assume, in addition, that $F$ is convex.
Then $x=1$ is a global minimum of $F$, i.e.,
\[
F(1)\le F(x)\qquad \text{for all }x>0.
\]
\end{Lemma}
\begin{proof}
Since $P$ is symmetric and $F$ is continuous and nonconstant,
we have
\[
F(x)=F\Big(\frac 1 x\Big), \qquad x>0.
\]
Suppose there exists $x_0>0$ such that $F(x_0)<F(1)$.
Then, also, $F(1/x_0)=F(x_0)<F(1)$.

Since $1$ lies between $x_0$ and $1/x_0$, there exists $\theta\in(0,1)$ such that
\[
1=\theta x_0+(1-\theta)\frac1{x_0}.
\]
By convexity,
\[
F(1)
\le \theta F(x_0)+(1-\theta)F(1/x_0)
=F(x_0)
<F(1),
\]
a contradiction.
\end{proof}



\begin{Remark} {For every real value of $c$, the~bilinear Equation \eqref{eq:F-bilinear-law} reduces,
 to a d'Alembert \mbox{Equation~\eqref{dal}} (by using Lemma \ref{lem:affine-reduction}). The~parameter $c$ parametrizes the family but does not
create new solution types.}
\end{Remark}

 \subsection{Reduction to Classical D’Alembert}\label{sec:reduction}

In this part, we show that the bilinear family \eqref{eq:F-bilinear-law}
reduces, after~a change of variables, to~the classical d'Alembert~equation.

\begin{Lemma}\label{lem:log-bilinear}
Assume \eqref{eq:F-bilinear-law} and define $G$ by \eqref{eq:def-G}.
Then for all $t,u\in\R$,
\begin{equation}\label{eq:G-bilinear}
  G(t+u)+G(t-u)=2G(t)+2G(u)+c\,G(t)G(u).
\end{equation}
\end{Lemma}

\begin{proof}
Let $x=e^t$ and $y=e^u$ in \eqref{eq:F-bilinear-law}.
Using $xy=e^{t+u}$ and $x/y=e^{t-u}$ and $G(t)=F(e^t)$
gives \eqref{eq:G-bilinear}.
\end{proof}

\begin{Lemma}\label{lem:affine-reduction}
Assume \eqref{eq:G-bilinear} for some constant $c\in\R$.
\begin{itemize}
\item[(i)] If $c\neq 0$ and
\[
H(t):=1+\frac{c}{2}G(t),
\]
then $H$ satisfies the classical d'Alembert equation
\begin{equation}\label{dal}
H(t+u)+H(t-u)=2H(t)H(u).
\end{equation}
\item[(ii)] If $c=0$, then \eqref{eq:G-bilinear} reduces to
\begin{equation}\label{eq:quadratic-fe}
G(t+u)+G(t-u)=2G(t)+2G(u).
\end{equation}
\end{itemize}
\end{Lemma}
\begin{proof}
\begin{itemize}
\item[(i)] If $c\neq 0$, substituting  $H(t)$ into   \eqref{dal} one obtains
\end{itemize}
\[
  H(t+u)+H(t-u)
  =2+\frac{c}{2}\big(G(t+u)+G(t-u)\big).
\]
From \eqref{eq:G-bilinear}, we have
\[
  2+\frac{c}{2}\big(2G(t)+2G(u)+c\,G(t)G(u)\big)
  =2+cG(t)+cG(u)+\frac{c^2}{2}G(t)G(u).
\]
On the other hand, we have
\begin{align*}
  2H(t)H(u)
  &=2\Big(1+\frac{c}{2}G(t)\Big)\Big(1+\frac{c}{2}G(u)\Big)
  \\&=2+cG(t)+cG(u)+\frac{c^2}{2}G(t)G(u),
\end{align*}
so \eqref{dal} holds.\\
\begin{itemize}
\item[(ii)] If $c=0$, \eqref{eq:G-bilinear} reduces directly to \eqref{eq:quadratic-fe}.
\end{itemize}
\end{proof}

We now determine the solutions in both~cases.

\begin{itemize}
    \item[(i)] Case $c\neq 0$. The~function
\(
H(t)
\)
satisfies \eqref{dal}.
If $F$ is continuous, then $H$ is continuous.
Since $F(x)=F(1/x)$, $H$ is even and $H(0)=1$.

Under standard regularity assumptions,
(see~\cite{Aczel,AczelDhombres, Kannappan, Kannappan2,StetkaerBook}), all even solutions of \eqref{dal}  with $H(0)=1$ are
\[
H(t)=\cosh(\alpha t)
\quad\text{or}\quad
H(t)=\cos(\alpha t),
\]
for some $\alpha\in\R$.
Equivalently,
\[
H(t)=\frac{e^{\lambda t}+e^{-\lambda t}}{2},
\]
with $\lambda\in\C$.
Substituting this into the definition of $F$, we obtain
\[
F(e^t)=G(t)=\frac{2}{c}\big(H(t)-1\big).
\]

\item[(ii)] If $c=0$, then \eqref{eq:G-bilinear} reduces to \begin{equation}
    G(t+u)+G(t-u)=2G(t)+2G(u)\qquad \text{for all }t,u\in\R.
  \end{equation}
with $G$ even and $G(0)=0$. \end{itemize}

\begin{Theorem}[\cite{Kannappan2}]\label{thrm4}
Suppose $G:\R\to\R$ satisfies
\[
G(t+u)+G(t-u)=2G(t)+2G(u).
\]
If $G$ is continuous, or~continuous at a point,
bounded on $[0,\delta)$ for some $\delta>0$,
bounded on a set of positive measure,
or measurable,
then
\[
G(t)=k t^2,
\qquad k\in\R.
\]
\end{Theorem}

Combining both cases yields the full~classification.

\begin{Theorem}\label{thm:classification}
The continuous solutions of
\[
F(xy)+F\Big(\frac xy\Big)=2F(x)+2F(y)+cF(x)F(y)
\]
are given as~follows:
\begin{itemize}
\item[(i)] If $c\neq 0$,
\[
F(e^t)=\frac{2}{c}\big(\cosh(\alpha t)-1\big)
\quad\text{or}\quad
F(e^t)=\frac{2}{c}\big(\cos(\alpha t)-1\big),
\qquad \alpha\in\R, \quad\alpha\ne 0.
\]
\item[(ii)] If $c=0$,
\[
F(x)=k(\ln x)^2,
\qquad k\in\R.
\]
\end{itemize}
\end{Theorem}

{\begin{proof}
Set $G(t)=F(e^t)$ and 
\(
H(t)=1+\frac{c}{2}G(t), 
\) for $c\neq 0$. 
By Lemma~\ref{lem:affine-reduction}(i), the~function $H$ satisfies the classical d'Alembert equation
with $H$ continuous, even, and~$H(0)=1$.
By the standard classification~\cite{Aczel,AczelDhombres,Kannappan,Kannappan2,Kuczma,Papp,StetkaerBook}, the~solutions are
\[
H(t)=\cosh(\alpha t) \quad \text{or} \quad H(t)=\cos(\alpha t), 
\]
giving 
the two~branches.

If $c=0$, the~equation reduces to Lemma~\ref{lem:affine-reduction}(ii), and~the result follows from \mbox{Theorem~\ref{thrm4}}.
\end{proof}}

 \begin{Proposition}\label{prop:sign-c}
For the hyperbolic branch in Theorem~\ref{thm:classification}(i), 
we have $F(x)\ge 0$ for all $x>0$ if and only if $c>0$.
\end{Proposition}

\begin{proof}
Since $\cosh(\alpha t)-1\ge 0$ for all $t$,
\[
F(e^t)=\frac{2}{c}\big(\cosh(\alpha t)-1\big)
\]
is nonnegative for all $t$ if and only if $\frac{2}{c}>0$,
that is, $c>0$.
\end{proof}

We now express the hyperbolic branch in $x$-coordinates and 
identify the parameter regime in which the solution admits 
a natural interpretation as a reciprocal cost~function.

\begin{Corollary}\label{cor:cost-interpretation}
Let $c>0$ and consider the hyperbolic branch
\[
F(e^t)=\frac{2}{c}\big(\cosh(\alpha t)-1\big),
\qquad \alpha\in\R.
\]
Then, in~$x$-coordinates,
\[
F(x)=\frac{1}{c}\big(x^{\alpha}+x^{-\alpha}-2\big),
\qquad x>0.
\]
Moreover:
\begin{itemize}
\item[(i)] $F(x)\ge 0$ for all $x>0$;
\item[(ii)] $F(1)=0$;
\item[(iii)] $F(x)=F(1/x)$;
\item[(iv)] if $\alpha\neq 0$, then $F(x)=0$ if and only if $x=1$.
\end{itemize}
In particular, for~$c>0$ and $\alpha\neq 0$, $F$ defines 
a reciprocal cost function on $\Rp$.
\end{Corollary}

 \begin{Corollary}
For every $c\in\R$, the~equation
\[
F(xy)+F\Big(\frac xy\Big)=2F(x)+2F(y)+cF(x)F(y)
\]
admits a continuous nonconstant solution
$F:\Rp\to\R$ satisfying $F(1)=0$ and $F(x)=F(1/x)$.
\end{Corollary}

\begin{proof}
The explicit solutions given in Theorem~\ref{thm:classification}
provide such functions for each $c\in\R$.
\end{proof}

\begin{Corollary}\label{cor:convex-case}
Under the assumptions of Lemma~\ref{lem:symmetry-convex-min},
assume in addition that $F$ is convex and that $c\neq 0$ in
\eqref{eq:F-bilinear-law}. Then $c>0$, and~only the hyperbolic branch is admissible, i.e.,
\[
F(e^t)=\frac{2}{c}\big(\cosh(\alpha t)-1\big).
\]
Moreover, $\alpha\ge 1$.
\end{Corollary}

\begin{proof}
By Lemma~\ref{lem:symmetry-convex-min}, $F(1)$ is a global minimum.
Since $F(1)=0$, we have $F(x)\ge 0$ for all $x>0$.
Because $c\neq 0$, Theorem~\ref{thm:classification}(i) holds.
The cosine branch
\[
F(e^t)=\frac{2}{c}\big(\cos(\alpha t)-1\big)
\]
is not convex on $\Rp$, since
\[
G''(t)=-\frac{2}{c}\alpha^2\cos(\alpha t)
\]
changes sign. Hence, $F''$ also changes sign, so $F$ cannot be convex. Therefore, only the hyperbolic branch remains
\[
F(e^t)=\frac{2}{c}\big(\cosh(\alpha t)-1\big).
\]
Since $F(x)\ge 0$ for all $x>0$, Proposition~\ref{prop:sign-c}
implies $c>0$.
Further, we write $t=\ln x$ and set $U(t):=\cosh(\alpha t)-1$.
Then $F(x)=\frac{2}{c}U(\ln x)$ and
\[
F''(x)=\frac{2}{c x^2}\big(U''(t)-U'(t)\big)
=\frac{2}{c x^2}\big(\alpha^2\cosh(\alpha t)-\alpha\sinh(\alpha t)\big).
\]
Convexity of $F$ on $\Rp$ means $F''(x)\ge 0$ for all $x>0$, i.e.,
\[
\alpha^2\cosh(\alpha t)-\alpha\sinh(\alpha t)\ge 0
\qquad\text{for all }t\in\R.
\]
Since $\cosh(\alpha t)=\cosh(|\alpha|t)$, we may assume without loss of generality that $\alpha>0$.
{(If $\alpha=0$, then $G\equiv 0$, contradicting nonconstancy; if $\alpha<0$, replace $\alpha$ by $|\alpha|$ since $\cosh$ is even.)}
Dividing by $\alpha>0$, we obtain
\[
\alpha\cosh(\alpha t)\ge \sinh(\alpha t)\qquad\text{for all }t\in\R.
\]
For $t>0$, it is equivalent to
\[
\alpha \ge \tanh(\alpha t).
\]
Since $\tanh(\alpha t) \to 1$ as $t \to \infty$,
the above inequality implies $\alpha \ge 1$. 
For $t<0$: since $\alpha>0$ and $\sinh(\alpha t)<0$,
the inequality $\alpha\cosh(\alpha t)\ge\sinh(\alpha t)$ holds trivially. For~$t=0$, both sides vanish.
Hence, the condition $\alpha\ge 1$ is both necessary and sufficient for all $t\in\R$.

Conversely, if~$\alpha \ge 1$, then
\[
\alpha \ge \tanh(\alpha t)
\qquad \text{for all } t,
\]
because $\tanh(\alpha t) < 1$ for every finite $t$.
Hence, $\alpha \ge 1$ is both necessary and sufficient
for global convexity.
\end{proof}

\section{D'Alembert Inevitability for \emph{n}-Dimensional~Cost}\label{sec:Ndim}

In this part, we extend the inevitability result to functions defined on 
$\Rp^n$.
Let $\mathbf{x}=(x_1,\dots,x_n)$ and $\mathbf{y}=(y_1,\dots,y_n)$
be elements of $\Rp^n$, with~
\[
\mathbf{x}\cdot\mathbf{y}=(x_1 y_1,\dots,x_n y_n),\qquad
\mathbf{x}/\mathbf{y}=(x_1/y_1,\dots,x_n/y_n),\qquad
\mathbf{1}=(1,\dots,1).
\]

We also write $\ln \mathbf{x} := (\ln x_1, \ldots, \ln x_n) \in \mathbb{R}^n$ and, 
for ${\boldsymbol{\alpha}} = (\alpha_1, \ldots, \alpha_n) \in \mathbb{R}^n$, use the~notation
\[
\mathbf{x}^{\boldsymbol{\alpha}} := \prod_{k=1}^{n} x_k^{\alpha_k}, 
\qquad
\boldsymbol{\alpha} \cdot \ln \mathbf{x} := \sum_{k=1}^{n} \alpha_k \ln x_k 
= \ln(\mathbf{x}^{\boldsymbol{\alpha}}).
\]

\begin{Definition}\label{def:ndim-poly-law}
A function $F:\Rp^n\to\R$ satisfies an n-dimensional polynomial composition law
if there exists a polynomial $P\in\R[u,v]$ such that for all $\mathbf{x},\mathbf{y}\in\Rp^n$,
\begin{equation}\label{eq:ndim-law}
F(\mathbf{x}\cdot\mathbf{y})+F(\mathbf{x}/\mathbf{y})=P\big(F(\mathbf{x}),F(\mathbf{y})\big).
\end{equation}
\end{Definition}

The algebraic classification of the polynomial combiner $P$
depends only on the functional equation and on the nondegeneracy of the range of $F$,
and therefore it is independent of the dimension $n$.
 
\begin{Theorem}\label{thm:ndim-bilinear}
Let $F:\Rp^n\to\R$ be continuous and nontrivial, with~$F(\mathbf{1})=0$, { where $\mathbf{1}=(1,\dots,1)$},
and suppose \eqref{eq:ndim-law} holds with a symmetric polynomial combiner $P$
of total degree at most two.
Then there exists $c\in\R$ such that
\[
P(u,v)=2u+2v+c\,uv.
\]
\end{Theorem}

\begin{proof}
Since $P$ is symmetric, from~\eqref{eq:ndim-law}, we get
\[
F(\mathbf{x}/\mathbf{y})=F(\mathbf{y}/\mathbf{x})
\qquad (\mathbf{x},\mathbf{y}\in\Rp^n).
\]
By substituting $\mathbf{y}=\mathbf{1}$, we obtain the reciprocity
$F(\mathbf{z})=F(\mathbf{z}^{-1})$ for all $\mathbf{z}\in\Rp^n$.

Now set $\mathbf{x}=\mathbf{1}$ in \eqref{eq:ndim-law}. Using $F(\mathbf{1})=0$
and reciprocity, we obtain
\[
P(0,F(\mathbf{y}))=F(\mathbf{y})+F(\mathbf{y}^{-1})=2F(\mathbf{y})
\qquad (\mathbf{y}\in\Rp^n).
\]
Since $F$ is continuous and nontrivial with $F(\mathbf{1})=0$,
its range contains a nondegenerate interval $I$ with $0\in I$.
Hence, the polynomial $v\mapsto P(0,v)-2v$ vanishes on $I$, so
\[
P(0,v)=2v\qquad \text{for all }v\in\R.
\]
By symmetry, also $P(u,0)=2u$ for all $u\in\R$.

Let us 
write a general symmetric quadratic polynomial
\[
P(u,v)=a+b(u+v)+c\,uv+e(u^2+v^2).
\]
Then
\[
P(0,v)=a+bv+ev^2=2v\quad \text{for all }v,
\]
so $a=0$, $b=2$, and~$e=0$. Therefore
\[
P(u,v)=2u+2v+c\,uv
\]
This completes the proof.
\end{proof}

We now classify the corresponding solutions. In~logarithmic coordinates $\mathbf{t}=(\ln x_1,\dots,\ln x_n)\in\R^n$,
define
\[ 
G(\mathbf{t})=F(e^{t_1},\dots,e^{t_n}), \qquad \mathbf{t}\in\R^n. \]
{ We treat $\mathbf{t}\in\R^n$ as a column vector.}
If $F(\mathbf{x})=F(\mathbf{x}^{-1})$, then $G$ is even,
\(
G(-\mathbf{t})=G(\mathbf{t}).
\)

Assume first that $c\neq 0$ and define
\(
H(\mathbf{t})=1+\frac{c}{2}G(\mathbf{t}).
\)
Then $H$ is continuous and satisfies the $n$-dimensional d'Alembert equation
\[
H(\mathbf{t}+\mathbf{u})+H(\mathbf{t}-\mathbf{u})
=
2H(\mathbf{t})H(\mathbf{u}),
\qquad \mathbf{t},\mathbf{u}\in\R^n,
\]
with $H(\mathbf{0})=1$ and $H$ even.\\

In the following theorem, we will classify the~solutions. 



\begin{Theorem}\label{thm:ndim-solutions}
All continuous solutions of
\begin{equation}\label{RCL}
F(\mathbf{x}\cdot\mathbf{y})+F(\mathbf{x}/\mathbf{y})
=2F(\mathbf{x})+2F(\mathbf{y})+c\,F(\mathbf{x})F(\mathbf{y}),
\qquad F(\mathbf{1})=0,
\end{equation}
are as~follows:
\begin{itemize}
\item[{(i)}] If $c\neq 0$, then there exists 
$\boldsymbol{\alpha}=(\alpha_1,\dots,\alpha_n)\in\R^n$
such that either
\[
F(\mathbf{x})
=
\frac{2}{c}\left(
\cosh\!\Bigl(\sum_{k=1}^n \alpha_k \ln x_k\Bigr)-1
\right)
=
\frac{1}{c}\left(
\prod_{k=1}^n x_k^{\alpha_k}
+
\prod_{k=1}^n x_k^{-\alpha_k}
-2
\right),
\]
or
\[
F(\mathbf{x})
=
\frac{2}{c}\left(
\cos\!\Bigl(\sum_{k=1}^n \alpha_k \ln x_k\Bigr)-1
\right).
\]

\item[{(ii)}] If $c=0$, then
\[
F(\mathbf{x})=\sum_{i,j=1}^n a_{ij}\,\ln x_i\,\ln x_j
\]
for some symmetric matrix $\mathbf{A}=(a_{ij})\in\R^{n\times n}$.
\end{itemize}
\end{Theorem}

\begin{proof}
\begin{itemize}
\item[(i)] Case $c\neq 0$. Since $F$ is continuous, $G(\mathbf{t})=F(e^{t_1},\dots,e^{t_n})$ and $H(\mathbf{t})=1+\frac{c}{2}\,G(\mathbf{t})$ are continuous with $H(\mathbf{0})=1$. { A direct computation shows that $H$ satisfies
\[
H(\mathbf{t}+\mathbf{u})+H(\mathbf{t}-\mathbf{u})=2H(\mathbf{t})H(\mathbf{u}),
\qquad \mathbf{t},\mathbf{u}\in\mathbb{R}^n.
\]
Thus, $H$ satisfies the classical d’Alembert functional equation on $\mathbb{R}^n$ in the vector notation used in this section. }
By the known classification of continuous solutions of the
$n$-dimensional d'Alembert equation  on $\R^n$
(see~\cite{Aczel,Kannappan2,StetkaerBook}),
there exists $\boldsymbol{\lambda}\in\C^n$ such~that
\[
H(\mathbf{t})
=\frac12\Big(e^{\boldsymbol{\lambda}\cdot\mathbf{t}}
+e^{-\boldsymbol{\lambda}\cdot\mathbf{t}}\Big)
=\cosh(\boldsymbol{\lambda}\cdot\mathbf{t}).
\]
If $H$ is real-valued for all $\mathbf{t}\in\R^n$, then
$\boldsymbol{\lambda}$ must be either real or purely imaginary.
Indeed, if~$\boldsymbol{\lambda}$ has both nonzero real and imaginary parts,
then $H(\mathbf{t})$ 
cannot remain real for all $\mathbf{t}\in\R^n$.
Writing $\boldsymbol{\lambda}=\boldsymbol{\alpha}\in\R^n$
or $\boldsymbol{\lambda}=i\boldsymbol{\alpha}$
gives the two real branches
\[
H(\mathbf{t})=\cosh(\boldsymbol{\alpha}\cdot\mathbf{t})
\quad\text{or}\quad
H(\mathbf{t})=\cos(\boldsymbol{\alpha}\cdot\mathbf{t}).
\]
Since $G(\mathbf{t})=\frac{2}{c}(H(\mathbf{t})-1)$ and $\mathbf{t}=\ln\mathbf{x}$,
we obtain
\[
F(\mathbf{x})
=\frac{2}{c}\Big(\cosh(\boldsymbol{\alpha}\cdot\ln\mathbf{x})-1\Big)
=\frac{1}{c}\Big(\mathbf{x}^{\boldsymbol{\alpha}}+\mathbf{x}^{-\boldsymbol{\alpha}}-2\Big),
\]
or
\[
F(\mathbf{x})
=\frac{2}{c}\Big(\cos(\boldsymbol{\alpha}\cdot\ln\mathbf{x})-1\Big).
\]

\item[(ii)] Case $c=0$. Then the equation for $G$ reduces to
\[
G(\mathbf{t}+\mathbf{u})+G(\mathbf{t}-\mathbf{u})=2G(\mathbf{t})+2G(\mathbf{u}),
\qquad \mathbf{t},\mathbf{u}\in\R^n.
\]
This is a Jensen-type quadratic functional equation on $\R^n$.
(see~\cite{AczelDhombres,Kannappan2}).
 {By the standard classification of continuous solutions of the quadratic Jensen-type equation on $\R^n$ (see~\cite{AczelDhombres,Kannappan2}), every solution has the form
\[
G(\mathbf{t})=\mathbf{t}^T \mathbf{A}\,\mathbf{t}, \qquad \mathbf{t}\in\R^n,
\]
where $\mathbf{A}=(a_{ij})\in\R^{n\times n}$ is a symmetric matrix.}
Finally, for~$\mathbf{t}=\ln\mathbf{x}$, we have
\[
F(\mathbf{x})=\sum_{i,j=1}^n a_{ij}\,\ln x_i\,\ln x_j,
\]
which completes the proof.\end{itemize}
\end{proof}

{ \begin{Remark}{For $F$ to serve as a cost function (non-negative with $F(\mathbf{x}) = 0$ only at $\mathbf{x} = \mathbf{1}$)~\cite{WZ1}, the~matrix $\mathbf{A}$ must be positive definite.
}
\end{Remark}}

\begin{Remark}{In the case $c\neq 0$, the~$\cosh$ branch with $c>0$
and $\boldsymbol{\alpha}\in\R^n$
satisfies
\(
F(\mathbf{x})\ge 0
\)
with equality if and only if $\mathbf{x}=\mathbf{1}$.
Hence, this branch is compatible with the interpretation
of $F$ as a cost~function.}
\end{Remark}

When $c\neq 0$,~Equation \eqref{eq:ndim-law} is very restrictive.
By Theorem~\ref{thm:ndim-solutions}, the~function $F$
depends on $\mathbf{x}\in\Rp^n$
only through the expression
\(
\boldsymbol{\alpha}\cdot\ln\mathbf{x}.
\)
Thus, even in dimension $n$, the~effective dependence
is~one-dimensional.

In applications, cost functions on $\Rp^n$ 
are assumed to be additively separable,
reflecting independent contributions of different coordinates.
It is therefore natural to ask whether such separable forms
are compatible with the composition law \eqref{RCL}.

Suppose that $F$ has a form
\begin{equation}\label{sep}
F(\mathbf{x})=\sum_{k=1}^n f_k(x_k).
\end{equation}
If each $f_k$ satisfies the composition law in one variable,
then
\[
F(\mathbf{x}\cdot\mathbf{y})+F(\mathbf{x}/\mathbf{y})
=
2F(\mathbf{x})+2F(\mathbf{y})
+
c\sum_{k=1}^n f_k(x_k)f_k(y_k).
\]
However,
\[
P(F(\mathbf{x}),F(\mathbf{y}))
=
2F(\mathbf{x})+2F(\mathbf{y})
+
cF(\mathbf{x})F(\mathbf{y})
\]
contains additional mixed terms of the form
\[
c\sum_{k\ne j} f_k(x_k)f_j(y_j),
\]
which cannot vanish unless at most one component is nontrivial.
This indicates that $F$ given by \eqref{sep} is incompatible
with the  composition law \eqref{RCL} when $c\neq 0$.

More precisely, the~following statement~holds.

\begin{Theorem}
Let $c\neq 0$ and $n\ge 2$.
If $F:\Rp^n\to\R$ satisfies 
\[
F(\mathbf{x}\cdot\mathbf{y})+F(\mathbf{x}/\mathbf{y})
=
2F(\mathbf{x})+2F(\mathbf{y})
+
cF(\mathbf{x})F(\mathbf{y}),
\]
then $F$ cannot be written in the form \eqref{sep}
with two or more nonconstant components.
\end{Theorem}

\begin{proof}
In logarithmic coordinates
\[
G(\mathbf{t})
=
\frac{2}{c}
\big(
\cosh(\boldsymbol{\alpha}\cdot\mathbf{t})-1
\big)
\quad\text{or}\quad
G(\mathbf{t})
=
\frac{2}{c}
\big(
\cos(\boldsymbol{\alpha}\cdot\mathbf{t})-1
\big).
\]
{We denote by $\mathbf{0}=(0,\dots,0)\in \mathbb{R}^n$ the zero vector. For~the hyperbolic branch, $G(\mathbf{t})=\frac{2}{c}(\cosh(\boldsymbol{\alpha}\cdot\mathbf{t})-1)$, and~for the trigonometric branch, $G(\mathbf{t})=\frac{2}{c}(\cos(\boldsymbol{\alpha}\cdot\mathbf{t})-1)$.} In either case, a~direct computation gives
\[
\frac{\partial^2 G}{\partial t_j\,\partial t_k}(\mathbf{0})
=
{\pm}\frac{2}{c}\,\alpha_j\alpha_k
\qquad (j\neq k),
\]
{where the sign $+$ corresponds to the hyperbolic branch and the sign $-$ to the trigonometric~branch.}

If $F(\mathbf{x})=\sum_k f_k(x_k)$,
then $G(\mathbf{t})=\sum_k g_k(t_k)$,
and therefore all mixed partial derivatives vanish.
Hence, $\alpha_j\alpha_k=0$ for all $j\neq k$.
 Thus, at most one component of $\boldsymbol{\alpha}$ is nonzero.
Consequently, $F$ depends on at most one coordinate,
so a decomposition with two nonconstant components is impossible.
\end{proof}

\begin{Corollary}
For $c\neq 0$, no additively separable cost with at least two nonconstant coordinate
components is compatible with the bilinear combiner.
\end{Corollary}

In the following example, we provide a realization of the multidimensional
rigidity result.
We construct a $16$-dimensional system depending on two
parameters $(r,s)$ and show that the induced reciprocal cost depends
only on a single scalar~aggregate.

\begin{Example}\label{ex:16dim}

Let $(r,s)\in\Omega\subset\mathbb{R}^2$ and define
\[
\boldsymbol{\Phi}(r,s)
=
(\phi_1(r,s),\dots,\phi_{16}(r,s))
\in\mathbb{R}^{16}
\]
by
\[
\boldsymbol{\Phi}(r,s)
=
\bigl(
r,\ s,\ r+s,\ r-s,\ rs,\ r^2,\ s^2,\ r^2-s^2,\ 2rs,\ 
r^3,\ s^3,\ r^2s,\ rs^2,\ r^4,\ s^4,\ r^2s^2
\bigr).
\]
Let $\boldsymbol{\alpha}\in\mathbb{R}^{16}$ and set
\[
S(r,s)
=
\boldsymbol{\alpha}\cdot\boldsymbol{\Phi}(r,s).
\]
Define
\[
\mathbf{x}(r,s)
=
\bigl(
e^{\alpha_1\phi_1(r,s)},\dots,
e^{\alpha_{16}\phi_{16}(r,s)}
\bigr)
\in\Rp^{16}.
\]
Then
\[
\sum_{k=1}^{16}\ln x_k(r,s)
=
S(r,s),
\qquad
\prod_{i=1}^{16} x_i(r,s)
=
e^{S(r,s)}.
\]
The reciprocal cost on $\Rp^{16}$ is given by
\[
F(\mathbf{x})
=
\frac12\Bigl(R+R^{-1}\Bigr)-1,
\qquad
R=\prod_{i=1}^{16} x_i.
\]
Under the above parametrization, this becomes
\[
F(r,s)
=
\cosh\!\bigl(S(r,s)\bigr)-1.
\]
The reciprocal cost depends only on the single scalar quantity
\[
S(r,s)
=
\boldsymbol{\alpha}\cdot\boldsymbol{\Phi}(r,s)
=\sum_{k=1}^{16}\ln x_k.
\]
By Theorem \ref{thm:ndim-solutions},  the system collapses to a logarithmic direction for $c\neq0$.
{In the case $c\neq 0$, the~function $F(r,s)$ depends only on the scalar quantity 
$S(r,s)=\boldsymbol{\alpha}\cdot\boldsymbol{\Phi}(r,s)$. 
Hence, for~any two points $(r_1,s_1)$ and $(r_2,s_2)$ such that 
$S(r_1,s_1)=S(r_2,s_2)$, we have 
$F(r_1,s_1)=F(r_2,s_2)$. 
Therefore, $F$ is constant along the level sets of $S$, and~the dependence on $(r,s)$ reduces effectively to one dimension. 
Here, the~level sets of $S$ are the sets 
\[
\{(r,s)\in \Omega : S(r,s)=\gamma\}, \quad \gamma\in\mathbb{R}.
\]}
The collapse and no-collapse regimes are illustrated in Figure~\ref{fig:collapse-vs-additive}.

\begin{figure}
\includegraphics[width=0.48\textwidth]{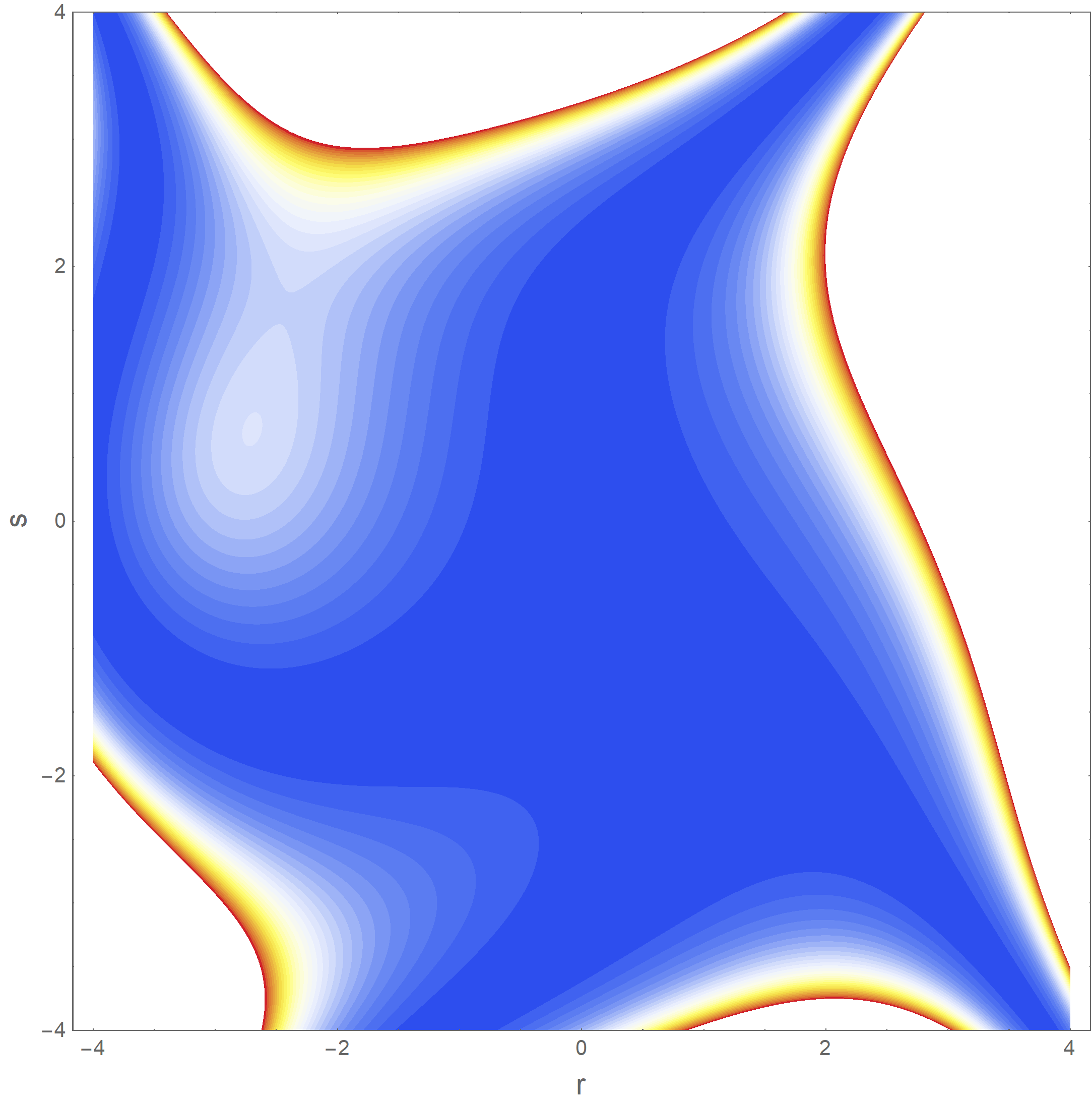}
\hspace{8pt}
\includegraphics[width=0.48\textwidth]{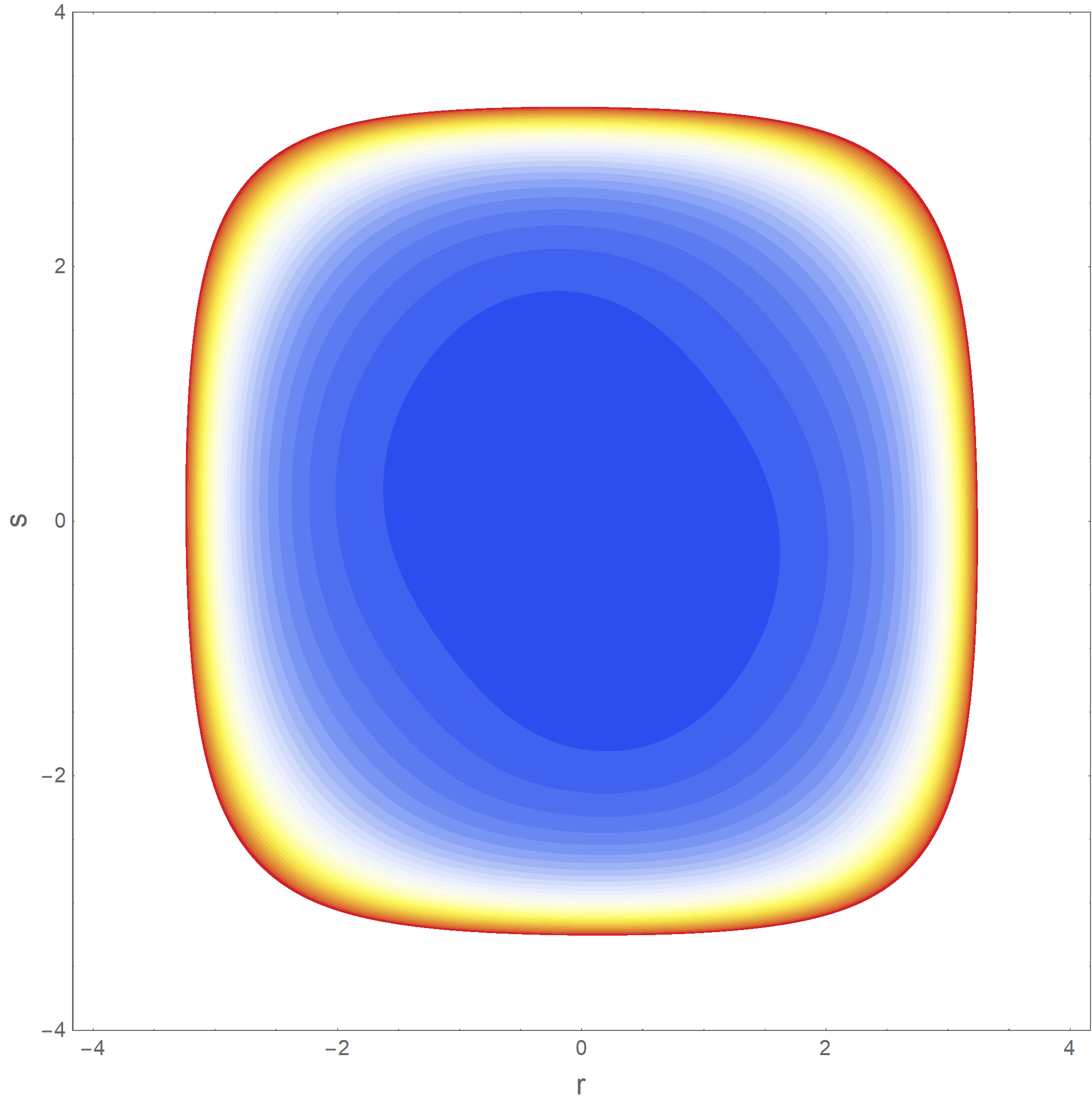}
\caption{(\textbf{Left}) {collapse} 
 regime ($c\neq 0$), where $F$ depends only on $S(r,s)$ and is constant on its level sets, reducing the dependence to one dimension. (\textbf{Right}) no-collapse regime, where $F$ depends genuinely on $(r,s)$.}
\label{fig:collapse-vs-additive}
\end{figure}
\end{Example}

\begin{Remark}\label{rem:ndim-separate}{If we consider a different combiner $P_k$ in each coordinate,
\[
F(\mathbf{x}\cdot_k \mathbf{y})
+
F(\mathbf{x}\mathbin{/_{k}}\mathbf{y})
=
P_k\big(F(\mathbf{x}),F(\mathbf{y})\big),
\]
where only the $k$-th component is modified,
then necessarily
\[
P_k(u,v)=2u+2v+c_k\,uv
\]
for each $k$.
The case $c_1=\dots=c_n$ reduces to the consideration above.
The compatibility of unequal parameters $c_k$ remains open.}
\end{Remark}

\section{Canonical Coefficient~Selection}\label{sec:calibration}

The previous sections establish that symmetry and the polynomial composition law
uniquely force the one-parameter bilinear family~\eqref{eq:F-bilinear-law}
for some real constant $c\in\R$.


We now show that a natural local normalization selects one distinguished member of this~family.

\begin{Definition}
Let $F : \mathbb{R}_{>0} \to \mathbb{R}$. 
The log-curvature of $F$, denoted $\kappa(F)$, is defined as
\[
\kappa(F) := \lim_{t \to 0} \frac{2 F(e^{t})}{t^{2}},
\]
provided this limit exists.     
\end{Definition}

{When the limit exists, $\kappa(F)$ is the quadratic scaling coefficient of $F(e^{t})$ at $t = 0$. 
This does not assume {a priori} that $F$ is $C^{2}$. 
The existence of the limit provides the required~regularity.}

By the change of variables $x = e^{t}$, the~limit exists if and only if
\[
\lim_{x \to 1} \frac{2 F(x)}{(\ln x)^{2}}
\]
exists, and~in that case, the two limits~coincide.

Assume that the limit  $\kappa(F)$ exists.
Then necessarily $F(1)=0$, since otherwise the quotient 
$\frac{2F(e^t)}{t^2}$ diverges as $t \to 0$. Set $G(t)=F(e^t)$. If~$G$ is twice differentiable at $0$, then
\[
\kappa(F)=G''(0).
\]
The calibration condition $\kappa(F)=1$ means
\[
G(t)=\tfrac12 t^2 + o(t^2)
\qquad (t\to 0).
\]

\medskip

We now determine how this calibration constrains
the parameter $c$.
\begin{Theorem}
Let $F$ be a continuous nonconstant solution of
\eqref{eq:F-bilinear-law} with $c\neq 0$.
Assume that $F$ belongs to the hyperbolic branch described in
Theorem~\ref{thm:classification}, that is,
\[
F(e^t)=\frac{2}{c}\big(\cosh(\alpha t)-1\big)
\quad \text{for some } \alpha>0.
\]
If $\kappa(F)=1$, then $c=2\alpha^2$.
\end{Theorem}
\begin{proof}
Let $G(t)=F(e^t)$. For~$c\neq 0$, define 
\(
H(t)=1+\frac{c}{2}G(t).
\)
By Lemma~\ref{lem:affine-reduction}(i), $H$ satisfies the classical
d'Alembert Equation \eqref{dal}. By~assumption, we are in the hyperbolic branch,
so
\[
H(t)=\cosh(\alpha t),
\qquad \alpha>0.
\]

Hence
\[
G(t)=\frac{2}{c}\big(H(t)-1\big)
=\frac{2}{c}\big(\cosh(\alpha t)-1\big).
\]
Using the Taylor expansion at $t=0$,
\[
\cosh(\alpha t)-1=\frac{\alpha^2 t^2}{2}+o(t^2),
\]
we compute
\[
\kappa(F)
=
\lim_{t\to 0}\frac{2F(e^t)}{t^2}
=
\lim_{t\to 0}\frac{2G(t)}{t^2}
=
\frac{4}{c}
\lim_{t\to 0}\frac{\cosh(\alpha t)-1}{t^2}
=
\frac{4}{c}\cdot\frac{\alpha^2}{2}
=
\frac{2\alpha^2}{c}.
\]
Therefore, the~calibration condition $\kappa(F)=1$
implies $c=2\alpha^2$.
\end{proof}

We now combine the d’Alembert Inevitability Theorem~\ref{thm:bilinear-forced}
with the solution classification in Theorem~\ref{thm:classification}
and the calibration condition.
\pagebreak
\begin{Theorem}\label{thm:calibrated-inevitability}
Assume:
\begin{itemize}
\item[(i)] $F$ is continuous and nonconstant;
\item[(ii)] $F$ satisfies the bilinear composition law ~\eqref{eq:F-bilinear-law};
\item[(iii)] $F$ is convex and nonnegative on $(0,\infty)$;
\item[(iv)] $\kappa(F)=1$.
\end{itemize}
Then $F$  belongs to the hyperbolic family
\begin{equation}\label{eq:hyperbolic-family}
F_\alpha(x)
=
\frac{1}{\alpha^2}
\big(
\cosh(\alpha \ln x)-1
\big),
\qquad \alpha\ge 1,
\end{equation}
and 
\[
c=2\alpha^2.
\]
\end{Theorem}

\begin{proof}
Passing to logarithmic coordinates $G(t)=F(e^t)$
and applying Lemma~\ref{lem:affine-reduction},
\mbox{Equation \eqref{eq:F-bilinear-law}} reduces to the classical d'Alembert equation.
By Theorem~\ref{thm:classification},
all continuous solutions are either hyperbolic or~trigonometric.

Since $F$ is convex and nonnegative on $(0,\infty)$, $F$ must belong to the hyperbolic branch,
so that
\[
G(t)=F(e^t)
=
\frac{2}{c}\big(\cosh(\alpha t)-1\big)
\quad \text{for some } \alpha>0.
\]
The log-curvature is
\[
\kappa(F)
=
\lim_{t\to 0}\frac{2F(e^t)}{t^2}
=
\frac{2\alpha^2}{c}.
\]
Using $\kappa(F)=1$ gives
\[
c=2\alpha^2.
\]
Substituting this relation into the expression for $F$
yields
\begin{equation}\label{rrr}
F_\alpha(e^t)
=
\frac{1}{\alpha^2}
\big(\cosh(\alpha t)-1\big).    
\end{equation}
Finally, global convexity of $F$ on $(0,\infty)$
is equivalent to $\alpha\ge 1$
by Corollary~\ref{cor:convex-case}.
\end{proof}
\medskip

The parameter $\alpha$ reflects a multiplicative
rescaling of the logarithmic coordinate.
The representation \eqref{rrr} can be written equivalently as
\[
F_\alpha(x)
=
\frac{1}{\alpha^2}
F_1(x^\alpha).
\]
Thus, $\alpha$ does not introduce a new structural type of solution;
it corresponds only to a rescaling of the coordinate $t=\ln x$.
Without loss of generality, we may therefore assume $\alpha=1$.

\begin{Corollary}\label{cor:canonical-cost}
Under the assumption of
Theorem~\ref{thm:calibrated-inevitability},
after normalization of the multiplicative coordinate,
the canonical representative is
\[
F(x)=\frac12(x+x^{-1})-1.
\]
\end{Corollary}

\begin{proof} If we set $\alpha=1$ in
\eqref{eq:hyperbolic-family},
we have $c=2$ and
\[
F(e^t)=\cosh(t)-1
=
\frac12(e^t+e^{-t})-1.
\]
Returning to multiplicative coordinates $x=e^t$
yields
\[
F(x)=\frac12(x+x^{-1})-1.
\]
\end{proof}

\begin{Remark} If $c=0$, then the composition law \eqref{eq:F-bilinear-law}
reduces to  the additive branch,
and the classification yields
\[
F(x)=k(\ln x)^2.
\]
In this case,
\[
\kappa(F)=2k,
\]
so the normalization $\kappa(F)=1$
forces $k=\tfrac12$,
giving
\[
F(x)=\tfrac12(\ln x)^2.
\]
This provides a unit-curvature solution
in the additive regime,
which lies outside the bilinear ($c\neq 0$)~family.
\end{Remark}

 \section{Conclusions}

In this paper, we studied continuous nonconstant functions 
\(F:\Rp\to\R\) satisfying a symmetric polynomial composition law
\[
F(xy)+F\Big(\frac xy\Big)=P(F(x),F(y)).
\]

We first considered the case of higher-degree polynomial combiners.
Theorem~\ref{thm:degree-exclusion} shows that symmetric combiners with $\deg P\ge3$
are incompatible with the functional equation under a 
non-cancellation condition.
In particular, the~ cubic case 
in Example~\ref{exmldeg3} admits no nonconstant continuous solutions. 
Consequently, only polynomial combiners of degree at most two
can admit nontrivial continuous~solutions.

In the quadratic case, the~combiner $P$ is necessarily of the form
\[
P(u,v)=2u+2v+c\,uv,\qquad c\in\R.
\]

We also showed that symmetry of \(P\) is equivalent to reciprocity \(F(x)=F(1/x)\),
and that the normalization \(F(1)=0\) implies
\(P(u,0)=2u\) and \(P(0,v)=2v\).
For a given continuous nonconstant solution $F$, the~combiner is~unique.

Passing to logarithmic coordinates reduces the composition law
to the classical d'Alembert functional equation.
This gives a complete classification of continuous solutions:
the hyperbolic and trigonometric branches,
and the quadratic logarithmic case when \(c=0\).

In the $n$-dimensional case, we showed that the classification of \(P\)
remains unchanged.
For \(c\neq0\), every solution depends only on a single scalar
combination \(\boldsymbol{\alpha}\cdot\ln\mathbf{x}\).
As a consequence, additive separability
\(F(\mathbf{x})=\sum_k f_k(x_k)\)
is impossible for \(n\ge2\).

Finally, we introduced the log-curvature calibration
\(
\kappa(F)
\)
and proved that, for~nonnegative convex solutions with \(c\neq0\)
and \(\kappa(F)=1\),
the solutions belong to the family
\[
F_\alpha(x)=\frac{1}{\alpha^2}\big(\cosh(\alpha\ln x)-1\big),
\qquad c=2\alpha^2.
\]
After normalization $\alpha=1$, this gives the canonical reciprocal cost
\[
F(x)=\frac12\bigl(x+x^{-1}\bigr)-1.
\]
Thus, the~canonical reciprocal cost is uniquely determined by the structural~constraints.

{Several natural questions remain open for further investigation.
These include the classification of asymmetric polynomial combiners,
the stability of the polynomial composition law in the Hyers--Ulam sense,
and the multidimensional case with distinct coordinate~parameters.}


\vspace{6pt}

\noindent{\bf Author contributions.} Conceptualization, J.W.; Methodology, J.W. and M.Z.; Software, J.W.; Validation, J.W., M.Z., and E.A.; Formal Analysis, M.Z. and J.W.; Investigation, J.W., M.Z., and E.A.; Resources,
J.W.; Writing-Original Draft Preparation, J.W.; Writing-Review and Editing, M.Z., J.W., and E.A.; Funding Acquisition, J.W.



\begin{thebibliography}{99}
 

 

\bibitem{Aczel}
Acz\'el, J. 
\textit{Lectures on Functional Equations and Their Applications};
Academic Press: New York, NY, USA, 1966.

\bibitem{AczelDhombres}
Acz\'el, J.; Dhombres, J. 
\emph{Functional Equations in Several Variables};
Cambridge University Press:{Cambridge, UK, New York, New Rochelle,
Melbourne Sydney} 
 1989.

\bibitem{Akkouchi}
Akkouchi, M. 
A note on d’Alembert’s functional equation.
\emph{Ann. Math. Blaise Pascal} \textbf{2001}, \emph{8}, 1--6.


\bibitem{Czerwik}
Czerwik, S. 
\emph{Functional Equations and Inequalities in Several Variables};
World Scientific: {Singapore}, 2002.

\bibitem{Davison}
Davison, T.M.K. 
D’Alembert’s functional equation and Chebyshev polynomials.
\emph{Ann. Acad. Paed. Cracov. Studia Math.} \textbf{2001}, \emph{4}, 31--38.

\bibitem{dAlembert1769}
d'Alembert, J. 
M\'emoire sur les principes de m\'ecanique.
\emph{Hist. Acad. Sci. Paris} \textbf{1769}, {278--286.} 


\bibitem{EbanksStetkaer}
Ebanks, B.; Stetk{\ae}r, H. 
d’Alembert’s other functional equation.
\emph{Publ. Math. Debr.} \textbf{2015}, \emph{87}, 319--349.


\bibitem{Hyers}
Hyers, D.H. 
On the stability of the linear functional equation.
\emph{Proc. Natl. Acad. Sci. USA} \textbf{1941}, \emph{27}, 222--224.

\bibitem{Kannappan}
Kannappan, P. 
On the Functional Equation $f(x + y) + f(x - y) = 2f(x)f(y)$.
 \emph{Am. Math. Mon.} \textbf{1965}, \emph{72}, 374--377.

\bibitem{Kannappan2}
Kannappan, P. 
{\it Functional Equations and Inequalities with Applications};
Springer: {Berlin/Heidelberg, Germany}, 2009.

\bibitem{Kuczma}
Kuczma, M. 
\emph{An Introduction to the Theory of Functional Equations and Inequalities},
2nd ed.; Birkh\"auser: Berlin, Germany, 2009.

\bibitem{Papp}
Papp, F.J. 
The D'Alembert functional equation.
\emph{Amer. Math. Mon.} \textbf{1985}, \emph{92}, 273--275.

\bibitem{Picard1922}
Picard, C.-E. 
Deux le\c{c}ons sur certaines \'equations fonctionnelles et la g\'eom\'etrie non-euclidienne.
\emph{Bull. Soc. Math. Fr}. \textbf{1922}, \emph{46}, 404--416,425--432.

\bibitem{Poisson1804}
Poisson, S. 
Du parall\'elogramme des forces.
\emph{Corresp. Sur L'\'Ecole Polytech.} \textbf{1804}, \emph{1}, 356--360.



\bibitem{StetkaerBook}
Stetk{\ae}r, H. 
\emph{Functional Equations on Groups};
World Scientific: {Singapore}, 2013. 

\bibitem{WZ1}
Washburn, J.; Zlatanović, M. Uniqueness of the Canonical Reciprocal Cost. \emph{Mathematics} \textbf{2026}, \emph{14}, 935. 



\end{thebibliography}
\end{document}